\documentclass[12pt, fleqn]{article}
\usepackage{authblk}
\usepackage[utf8]{inputenc}
\usepackage[margin=1in]{geometry}
\usepackage{titling}
\usepackage{xcolor}

\usepackage{amsthm}
\usepackage{amssymb}
\usepackage{amsmath}
\usepackage{enumitem}
\numberwithin{equation}{section}
\newtheorem{theorem}{Theorem}[section]
\newtheorem{cor}{Corollary}[theorem]
\newtheorem*{theorem-non}{Theorem}
\newtheorem{lem}[theorem]{Lemma}
\newtheorem{defn}[theorem]{Definition}

\theoremstyle{remark}

\theoremstyle{remark}
\newtheorem*{remark}{Remark}
\newtheorem{rmk}{Remark}

\title{The Clifford Theory for Modular Representations of Finite Groups}
\author{Devjani Basu}

\affil{Department of Mathematics, Southern Illinois University, Carbondale, IL, U.S.A}
\date{} 

\begin{document}
\maketitle



\abstract 
Clifford theory  establishes a relation between the representation theory of a finite group and its normal subgroups. In this paper, we establish the Clifford theory for the modular representations of finite groups. The proofs are based on an explicit analysis of the representation spaces and their decompositions. We also analyze the relation between the modular representations of $SL_2(\mathbb{F}_p)$ in defining characteristic, with that of $GL_2(\mathbb{F}_p)$ using the modular Clifford theory.

\medskip


\section*{Introduction}

In 1898, it was known to Frobenius that the restriction of an irreducible ordinary character of a finite group to its normal subgroup is a sum of characters that are \emph{conjugate} to a given one. Here, ordinary characters are the characters defined over a field of characteristic $0$. In 1937, A. Clifford reproved this and extended the results over all fields and for infinite groups in his paper \cite{Clif}. 

In 1907, L.E. Dickson first studied the modular representations; these are the representations over a vector space over a field of nonzero characteristic $p$, such that $p$ divides the order of the group. As a consequence, the semisimplicity of the group algebra is lost; semisimplicity ensures that all representations are semisimple, i.e. are direct sums of irreducible representations. The theory of these representations were later extensively developed by R. Brauer \cite{BN} in 1941. In \cite{navcharacters}, Navarro provides a detailed exposition of this theory in the light of \emph{Brauer characaters} as characters of the modular representation, which are also known as modular characters, including some results on the Clifford theory.

 In this paper, we have developed the Clifford theory for the modular representations of a finite group, employing an explicit analysis of the representation spaces and their decompositions. The techniques are inspired by the expository paper \cite{CT} of Ceccherini-Silberstein, Scarabotti, and Toll, and gives a functional flavor, making it more suitable for applications in a harmonic analytical framework. 
 
Among other applications, this Clifford theory enables us to study the modular representations of finite groups of Lie type, such as $SL_n(\mathbb{F}_q)$, via the representations of $GL_n(\mathbb{F}_q)$, which is the finite group of fixed points of the connected reductive group $GL_n$. This owes to the fact that $SL_n$ is a normal subgroup of $GL_n$.

 In Sec.~1, we present the main definitions and proof the two main theorems of the Clifford theory for modular representations. The first one, Theorem \ref{thm:1}, shows the decomposition of the irreducible representation space of the group into the direct sum of irreducible representation spaces of the normal subgroup, upon restriction. The second main theorem, Theorem \ref{thm:2}, is a stronger version of the first one, strengthened via the \emph{inertia subgroup}.

Then, in Sec.~2, we exhibit how the Clifford theory connects the modular representations of $SL_2(\mathbb{F}_3)$ and $GL_2(\mathbb{F}_3)$ in characteristic $p=3$, using their corresponding modular (or Brauer) character tables. Then we proceed to describe the modular representations of $SL_2(\mathbb{F}_p)$ and  $GL_2(\mathbb{F}_p)$, in defining characteristic, explicitly, as homogeneous polynomial spaces in two variables. While the theorems \ref{thm:22} and \ref{thm: 23} establish the irreducibilty of these vector spaces, the relations \ref{eq:1}, \ref{eq:2}, and \ref{eq:3} establish the correspondence between these representations via the modular Clifford theory.


\section{Modular Clifford Theory}

Let $\mathbb{F}$ be an algebraically closed field of characteristic $p$, where $p$ is a prime.

Let $G$ be a finite group and $N\unlhd G$ be a normal subgroup of $G$ such that $p$ divides the order of $N$ and hence that of $G$.  

Here $\hat{G}$ (respectively $\hat{N}$) denotes the set of all equivalence classes of irreducible modular representations of $G$ (respectively $N$) over $\mathbb{F}$. For $\sigma \in \hat{N}$, let $V_\sigma$ be the representation space of $\sigma$ over $\mathbb{F}$.

The following definitions are parallel to those in \cite{CT} except that all the representations mentioned here are modular in nature.

\begin{defn}
Let $\sigma \in \hat{N}$ be a finite-dimensional representation of $N$. The \textbf{induced representation} $Ind^G_N\sigma$ of $G$ on  $V_\sigma$ of dimension [G:N] is defined as
\begin{equation*}
    Ind^G_N(V_{\sigma})=\{f:G \rightarrow V_{\sigma}|f(gn)=\sigma(n^{-1})f(x), \forall n\in N, g \in G \}
\end{equation*}
where $G$ acts on $V_\sigma$ by left translation.
\end{defn}

It is to be noted that Frobenius reciprocity does not neccesarily hold for modular representations of finite groups. For our purpose, we will heavily rely on Nakayama's formulas. 

\begin{lem}[Nakayama]

\label{Lem:Nak}

 Let $G$ be a finite group and $H \leq G$ be a subgroup of $G$. Let $\hat{G}$ (respectively $\hat{H}$) denotes the set of all equivalence classes of irreducible modular representations of $G$ (respectively $H$) over $\mathbb{F}$. For $\sigma \in \hat{H}$, let $V_\sigma$ be the representation space of $\sigma$, and for $\theta \in \hat{G}$, let $U_\theta$ be the representation space of $\theta$, over $\mathbb{F}$. Then
\begin{align}
    Hom_G(U_{\theta}, Ind_N^GV_{\sigma}) &\cong Hom_H(U_{\theta}, V_{\sigma})\\
     Hom_G(Ind_N^GV_{\sigma}, U_{\theta}) &\cong Hom_H(V_{\sigma}, U_{\theta})
     \end{align}
for $\theta \in \hat{G}$ and $\sigma \in \hat{H}$.
\end{lem}

 Let $\sigma < \rho$ denotes that $\sigma$ is a subrepresentation of $\rho$. 
\begin{defn}
\label{defn: Ghat}
For $\sigma \in \hat{N}$, $\hat{G}(\sigma)$ is defined as the set
\begin{center}
     $\hat{G}(\sigma)=\{\theta \in \hat{G} \ | \ V_{\sigma} \subseteq Res_{N}^{G}(U_{\theta})\}$
\end{center}
where $U_\theta$ be the representation space of $\theta$ over $\mathbb{F}$.
\end{defn}

\begin{defn}
For a fixed $g\in G$, the \textbf{g-conjugate of} $\mathbf{\sigma}$ in $\hat{N}$, denoted by $^{g}\sigma$, is defined as
\begin{equation}
\label{conjugacy}
^{g}\sigma(n)=\sigma(g^{-1}ng) 
\end{equation}
for all  $n \in N$.
\end{defn}

Here, the definition (\ref{conjugacy}) defines an action of $G$ on $\hat{N}$, such that for all $g_1$, $g_2$ $\in$ $G$ and $n \in N$,
\begin{center}
    $^{g_1g_2}\sigma(n)=\sigma(g_1g_2ng^{-1}_2g^{-1}_1)=^{g_1}\sigma(g_2ng^{-1}_2)=^{g_1}(^{g_2}\sigma(n)$),
\end{center}
i.e.
\begin{center}
    $^{g_1g_2}\sigma=^{g_1}(^{g_2}\sigma)$.
\end{center}

\noindent Let $\sigma \sim \rho$ denotes that the representations are equivalent, and $\{V_{\sigma}\}$ the equivalence class of representation spaces.

\begin{defn}
\label{inertia}
For $\sigma \in \hat{N}$, the \textbf{inertia group} of $\sigma$ is defined by the subgroup
\begin{center}
    $I_{G}(\sigma)=\{g\in G \ | \ ^{g}\sigma\sim \sigma\}<G$
\end{center}
i.e. for $g \in I_{G}(\sigma)$, $^{g}\{V_\sigma\} = \{V_{^{g}\sigma}\} =\{V_\sigma \}$.
\end{defn}

\begin{remark}
$I_G(\sigma)$ is the stabilizer of $\sigma$ in $G$. Furthermore, if $n_1$, $n$ $\in N$, we have
\begin{center}
    $^{n_1}\sigma(n)=\sigma(n_1)^{-1}\sigma(n)\sigma(n_1)$\\
    i.e. $^{n_1}\sigma\sim\sigma$.
\end{center}
Therefore, $N$ is normal in $I_G(\sigma)$ and $[G:I_G(\sigma)]$ is the number of $g$-conjugates of $\theta$ for all $g \in G$.
\end{remark}

Let $R$ be the set of representatives for the left $I_G(\sigma)$ cosets in $G$. Then, we have the coset decomposition
\begin{center}
    $G=\sqcup_{r\in R}rI_G(\sigma)$
\end{center}
and
\begin{center}
    $\{^g\sigma:g\in G\} = \{^r\sigma:r\in R\}$
\end{center}
where the represenation $^r\sigma$'s are pairwise nonequivalent. Let $Q$ be the set of representatives of the left $N$-cosets in $I_G(\sigma)$, such that 
\begin{center}
    $I_G(\sigma)=\sqcup_{q\in Q}qN$.
\end{center}
Then the coset decomposition of $G$ over $N$ is
\begin{center}
    $G=\sqcup_{r\in R}rI_G(\sigma)=\sqcup_{r\in R}\sqcup_{q\in Q}rqN=\sqcup_{t\in T}tN$
\end{center}
where $T=RQ$.

These definitions lead us to the first main theorem of this section on the decomposition of  $Res_N^G U_{\theta}$. We will often use $Res_N^G\theta$ for $Res_N^G U_{\theta}$ in this paper.

\begin{theorem}
\label{thm:1}
For $\sigma \in \hat{N}$, let $d$ be the index of $N$ in $I_G(\sigma)$. For $\theta\in \hat{G}(\sigma)$, let $\ell$ be the multiplicity of $\sigma$ in $Res_N^G\theta$, then we have
\begin{enumerate}[leftmargin=*]
    \item  The decomposition of $Res^G_N(Ind^G_N\sigma)$ into irreducible inequivalent subrepresentations is given by
    \begin{equation}
    Res^G_N(Ind^G_N\sigma)=\bigoplus_{t\in T} {^t\sigma}=d\bigoplus_{r\in R} {^r\sigma}
    \end{equation}
where $T$ and $R$ are as above.
    \item 
    $Hom_G(Ind_N^G\sigma, Ind_N^G\sigma)\cong \mathbb{F}^d$
    
    \item The decomposition of $Res_N^G\theta$ is given by
    \begin{equation}
        Res_N^G\theta = \ell\bigoplus_{r\in R}{^r\sigma}.
    \end{equation}
\end{enumerate}
\end{theorem}
\begin{proof}
\begin{enumerate}[leftmargin=*]
    \item 
    Let $U$ denote the representation space of $\sigma.$ For $t\in T$, $Z_t$ define the space of all functions $f:G \rightarrow U$ such that
    \begin{center}
        $f(t'n)=\delta_{t,t'}\sigma(n^{-1})f(t)$ 
    \end{center}
where $n\in N$ and $t'\in T$.
Then, we have $Ind^G_NU=\bigoplus_{t\in T}Z_t$.

Suppose the map $L_t:U\rightarrow Z_t$ is given by
\begin{center}
    $[L_tu](t'n)=\delta_{t,t'}\sigma(n^{-1})u$, 
\end{center}
$\forall u \in U$. Then, it is easy to see that $L_t$ is linear due to the linearity of the vector space, and injective because of the irreducibility of $\sigma$. Moreover, for any $f\in Z_t$ such that $f(t)=u\in U$, we have 
\begin{center}
    $[L_tu](t'n)=\delta_{t,t'}\sigma(n^{-1})u=\delta_{t,t'}\sigma(n^{-1})f(t)=f(t'n)$
\end{center}
Therefore, $L_t$ is surjective and hence it is a linear isomorphism.

Furthermore, suppose $\lambda=Ind_N^G\sigma$, then we have
\begin{align*}
 \lambda(n)L_tu](t_1n_1) &=[L_tu](n^{-1}t_1n_1) \\
    &=[L_tu](t_1t_1^{-1}n^{-1}t_1n_1) \\
    &=\delta_{t,t_1}\sigma((t_1^{-1}n^{-1}t_1n_1)^{-1})u \\
   &=\delta_{t,t_1}\sigma(n_1^{-1}t_1^{-1}nt_1)u \\
 & =\delta_{t,t_1}\sigma(n_1^{-1})\sigma(t_1^{-1}nt_1)u \\
  &=\delta_{t,t_1}\sigma(n_1^{-1})[(^{t_1}\sigma(n)u]\\
&  =[(L_t(^{t_1}\sigma(n)u)](t_1n_1)
\end{align*}

for all $u \in U$, $t_1 \in T$, and  $n, n_1 \in N$.
Therefore, $\lambda(n)L_t=L_t{^{t_1}\sigma(n)}$, that is, $L_t$ is an intertwiner inducing an equivalence between $(Res_N^G \lambda, Z_t)$ and $(^t\sigma, U )$.

Hence, we can write, 
\begin{equation*}
    Res^G_N(Ind^G_N \sigma) = \bigoplus_{t\in T} {^t}\sigma=\bigoplus_{r\in R}\bigoplus_{q\in Q} {^{rq}\sigma}=|Q|\bigoplus_{r\in R} {^r\sigma}=d\bigoplus_{r\in R} {^r\sigma},
\end{equation*}
where $Q$ is defined as above and $|Q|=[I_G(\sigma):N]=d$.
\item By (1), the multiplicity of $\sigma \in \hat{N}$ in $Res^G_NInd^G_N \sigma$ is $d$. Therefore, 
    $dim_{\mathbb{F}}Hom_N(\sigma, Res_N^GInd_N^G\sigma)=d$,  i.e. $Hom_N(\sigma, Res_N^GInd_N^G\sigma)\cong \mathbb{F}^d$.

Now, by Lemma \ref{Lem:Nak},
\begin{center}
    $Hom_G(Ind^G_N\sigma, Ind_N^G\sigma) \cong Hom_N(\sigma, Res_N^GInd_N^G\sigma) \cong \mathbb{F}^d$.
\end{center}
\item Let us define a map from $Hom_N(\sigma, Res_N^G\theta)$ to $Hom_N(^g\sigma, Res_N^G\theta)$ by $T\mapsto\theta(g)T$ for $g\in G$.

Note that for $S, T \in Hom_N(\sigma, Res_N^G\theta)$ and a scalar $\alpha$,
\begin{center}
    $(\alpha S+T)(u)=\alpha S(u)+T(u)=\theta(g)S(u)+\theta(g)T(u)=\theta(g)(\alpha S+T)(U)$,
\end{center}
by the application of vector properties and the fact that $\theta$ is an automorphism. Moreover, for $T\in Hom_N(\sigma, Res_N^G\theta)$,
\begin{equation*}
    \theta(g)T(u)=\theta(g)u \implies \theta(g)T(u)-\theta(g)u=0 \implies \theta(g)(T(u)-u)=0 \\
    \implies T(u)=u.
\end{equation*}
Therefore, the above map is linear and injective.

Furthermore, for any $T' \in Hom_N(^g\sigma, Res_N^G\theta)$, taking $T=\theta(g^{-1})T'$, we have 
\begin{center}
    $ T(n \cdot v)=\theta(g^{-1})T'(n\cdot v)=n\theta(g^{-1})T'(v)$.
\end{center}
Hence, we have $T \in Hom_N(\sigma, Res_N^G\theta)$. Also, $T \mapsto \theta(g)T=\theta(g)\theta(g^{-1})T'=T'$, i.e. $T \mapsto T'$. Therefore, the defined map is a linear isomorphism and $^g\sigma$ also has multiplicity $\ell$ in $Res_N^G \theta$.

By lemma \ref{Lem:Nak}, we have that $\theta$ has multiplicity $\ell$  of $\theta$.
Therefore, every irreducible subrepresentation in $Res^G_N\theta$ is an irreducible subrepresentation in $Res^G_NInd^G_N\sigma$. But,  $Res^G_NInd^G_N\sigma$ contains only irreducible representation of the form $^r\sigma$, $ r\in R$. Hence, we have
\begin{equation*}
    Res^G_N\theta=\ell\bigoplus_{r\in R}{^r\sigma}.
\end{equation*}
\end{enumerate}
\end{proof}

\begin{cor}
For any $\sigma \in \hat{N}$, then $Ind^G_N\sigma$ is irreducible if and only if $I_G(\sigma)=N$. Furthermore, for $\sigma_1$, $\sigma_2$ $\in \hat N$ such that $I_G(\sigma_1)=I_G(\sigma_2)=N$, $Ind^G_N\sigma_1 \sim Ind^G_N\sigma_2$ if and only if $\sigma_1$ and $\sigma_2$ are conjugate.
\end{cor}

\begin{proof}
We have $I_G(\sigma)=N \Rightarrow d=[I_G(\sigma):N]=1$. Now from (2) of the theorem,
$Hom_G(Ind_N^G\sigma, Ind_N^G\sigma)\cong \mathbb{F}^d=\mathbb{F}$, which implies $Ind_G^N\sigma$ is irreducible by Schur's lemma.

Now, from (1) of the theorem, setting $\theta=Ind^G_N\sigma_1$, which is irreducible, we have 
\begin{equation*}
    Res^G_N\theta=\bigoplus_{t\in T}{^t\sigma_1}=\bigoplus_{r\in R}{^r\sigma_1}.
 \end{equation*}
Given, $Ind^G_N(\sigma_1)\sim Ind^G_N(\sigma_2)$, i.e. $\theta \sim Ind^G_N(\sigma_2)$ implies that
\begin{equation*}
    Hom(Ind^G_N\sigma_2, \theta)\cong \mathbb{F} 
    \Leftrightarrow Hom(\sigma_2, Res^G_N\theta)\cong \mathbb{F} 
    \Leftrightarrow Hom(\sigma_2,\bigoplus_{r\in R}{^r\sigma_1} )\cong \mathbb{F}  \Leftrightarrow \sigma_2 \sim  {^r\sigma_1} 
\end{equation*}
for some $r\in R $. Hence, they are conjugates.
\end{proof}

\begin{remark}
The number $\ell= dim_{\mathbb{F}}Hom(\sigma, Res^G_N\theta)$ is called the \textbf{inertia index} of $\theta \in \hat{G}(\sigma)$ with respect to $N$.
\end{remark}

Now,  if $I$ denote the the inertia group  $I_G(\sigma)$ for $\sigma \in \hat{N}$, let $ \hat{I}(\sigma)$ denote the set
\begin{equation*}
   \hat{I}(\sigma)=\{ \phi \in \hat{I} \ | \ W_\phi \subseteq Ind^I_N V_\sigma\}
 \end{equation*}
where $W_\phi$ is the representation space of $\phi \in \hat{I}$, such that $Ind_N^I\sigma$ has the decomposition
\begin{equation}
    Ind^I_N V_\sigma=\bigoplus_{\phi \in \hat{I}(\sigma)} W^{m_{\phi}}_{\phi}  = \bigoplus_{\phi \in \hat{I}(\sigma)} m_{\phi}\phi
\end{equation}
where  $m_\phi$ is the multiplicity of $\phi \in \hat{I}(\sigma)$. Then there exists the follwing decomposition.

\begin{lem}
\label{lem:1}
\begin{enumerate}[leftmargin=*]
    \item  The decomposition of $Ind^G_N\sigma$  into irreducible modular representations of $G$ is
\begin{equation}
    Ind^G_N\sigma=\bigoplus_{\phi \in \hat{I}(\sigma)}m_{\phi}Ind_I^G\phi.
\end{equation}

\item If $\theta \in \hat{G}(\sigma)$, then $\theta= Ind_I^G \phi$ for some unique $\phi \in \hat{I}(\sigma)$.
\end{enumerate}
\end{lem}

\begin{proof}
1. Theorem \ref{thm:1}(2) implies that for $\sigma \in \hat{N}$
\begin{eqnarray*}
  \mathbb{F}^d & \cong & Hom_G(Ind_N^G\sigma, Ind_N^G\sigma) \\
        & = & Hom_G(Ind^G_I Ind_N^I\sigma, Ind^G_I Ind_N^I\sigma) \\
        & = & Hom_I(Ind_N^I\sigma, Res^G_I Ind^G_I Ind_N^I\sigma) \\
        & = & Hom_I(Ind_N^I\sigma, Ind_N^I\sigma),
\end{eqnarray*}
therefore, $I$ is also the inertia group of $\sigma$ in $I$.

Moreover, using (0.4), we have
\begin{eqnarray*}
d & = & dim_{\mathbb{F}} Hom_I(Ind_N^I\sigma, Ind_N^I\sigma) \\
  & = & dim_{\mathbb{F}} Hom_I(\bigoplus_{\phi \in \hat{I}(\sigma)}m_{\phi}\phi, \bigoplus_{\phi \in \hat{I}(\sigma)}m_{\phi}\phi) \\
  & = & dim_{\mathbb{F}}(Mat_{m_{\phi}\times m_{\phi}}(\mathbb{F})) \\
  & = & \sum_{\phi \in \hat{I}(\sigma)} m_\phi^2
\end{eqnarray*}
and also,
\begin{eqnarray*}
 d & = & dim_{\mathbb{F}} Hom_G(Ind^G_I (Ind_N^I)\sigma, Ind^G_I (Ind_N^I)\sigma)\\
 & = & dim_{\mathbb{F}}Hom_G(Ind^G_I (\bigoplus_{\phi \in \hat{I}(\sigma)}m_{\phi}\phi), Ind^G_I(\bigoplus_{\phi \in \hat{I}(\sigma)}m_{\phi}\phi))\\
 & = & dim_{\mathbb{F}}Hom_G(\bigoplus_{\phi \in \hat{I}(\sigma)}m_{\phi}Ind^G_I\phi, \bigoplus_{\phi \in \hat{I}(\sigma)}m_{\phi}Ind^G_I\phi)\\
 \Rightarrow \sum_{\phi \in \hat{I}(\sigma)} m_\phi^2 & = & \sum_{\phi, \psi \in \hat{I}(\sigma)}m_\phi m_\psi dim Hom_G(Ind_N^G \phi, Ind_N^G \psi)
\end{eqnarray*}

since $\phi$, $\psi$ are not equivalent. This implies
\begin{eqnarray*}
dim Hom_G(Ind_N^G \phi, Ind_N^G \psi) = \delta_{\phi, \psi} 
\Rightarrow   Hom_G(Ind_N^G \phi, Ind_N^G \phi) = \mathbb{F}
\end{eqnarray*}
for $\phi \in \hat{I}(\sigma)$. Therefore,
\begin{eqnarray*}
 Ind^G_N\sigma = Ind^G_I(Ind^I_N\sigma)= Ind^G_I(\bigoplus_{\phi \in \hat{I}(\sigma)}m_{\phi}\phi)= \bigoplus_{\phi \in \hat{I}(\sigma)}m_{\phi}Ind^G_I\phi.
\end{eqnarray*}

2. Let $\theta \in \hat{G}(\sigma)$, then by lemma \ref{Lem:Nak} 
\begin{equation*}
\theta  < Ind^G_N (\sigma) \Rightarrow \theta < \bigoplus_{\phi \in \hat{I}(\sigma)}m_{\phi}Ind^G_I\phi \Rightarrow \theta = Ind^G_I\phi
\end{equation*}
for some $\theta \in \hat{I}(\sigma)$, uniquely.
\end{proof}

\begin{theorem}[Clifford correspondence for modular representations]
\label{thm:2}
For $\sigma \in \hat{N}$, let $\hat{I}(\sigma)$ be defined as in Lemma \ref{lem:1}, then the map from $\hat{I}(\sigma)$ to $\hat{G}(\sigma)$, defined by
\begin{equation*}
    \phi \mapsto Ind^G_I(\phi)
\end{equation*}
is a bijection. Furthermore, the inertia index of $\phi \in \hat{I}(\sigma)$ with respect to $N$ as well as that of $Ind^G_I\phi$ is equal to $m_{\phi}$, the multiplicity of $\phi$ in $Ind^I_N(\sigma)$ and 
\begin{equation*}
    Res^I_N\phi=m_{\phi}\sigma.
\end{equation*}
\end{theorem}
\begin{proof}Let $\theta_1$, $\theta_2$ $\in \hat{G}(\sigma)$. Then by lemma \ref{Lem:Nak}, we have, 
\begin{equation*}
\theta_1, \theta_2 < Ind^G_N\sigma \Rightarrow \theta_1, \theta_2  < \bigoplus_{\psi\in \hat{I}} Ind^G_I\phi.
\end{equation*}
Then, there exist some $\phi_1, \phi_2 \in \hat{I}(\sigma)$ such that $\theta_1=Ind^G_I\phi_{1}$ and $\theta_2=\phi_{2}$. If $\theta_1 =\theta_2$, it implies that $\phi_1 =\phi_2$. Hence, the map is one-to-one.

Now, by Lemma \ref{lem:1} for any $\theta \in \hat{G}(\sigma)$, there exists a unique $\phi \in \hat{I}(\sigma)$ such that $\theta = Ind^{G}_I\phi$. Therefore, the map is onto and the map $\phi \mapsto Ind^G_I\phi$ is a bijection.

The inertia index of $\phi \in \hat{I}(\sigma)$ with respect to $N$ is given by $dimHom_N(\sigma, Res^{I}_N)\phi)$. By lemma \ref{Lem:Nak}, we have,
\begin{eqnarray*}
dimHom_N(\sigma, Res^{I}_N\phi) & = & dimHom_N(Ind^I_N\sigma, \phi)\\
& = & dimHom_N(\bigoplus_{\phi \in \hat{I}(\sigma)}m_{\phi}\phi, \phi)\\
& = & m_{\phi}dimHom_N(\phi, \phi)\\
& = & m_{\phi}.
\end{eqnarray*}
Furthermore, the inertia index of $Ind^G_I\phi \in \hat{G}(\sigma)$ with respect to N is $m_{\phi}$, since
\begin{eqnarray*}
dimHom_N(\sigma, Res^{I}_NInd^G_I\phi) & = & dimHom_N(Ind^I_N\sigma, Ind^G_I\phi)\\
& = & dimHom_N(\bigoplus_{\phi \in \hat{I}(\sigma)}m_{\phi}Ind^G_I\phi, Ind^G_I\phi)\\
& = & m_{\phi}.
\end{eqnarray*}
Now replacing $G$ with $I$ and $\theta$ with $\phi \in \hat{I}(\sigma)$ in Theorem \ref{thm:1}(3), we get
\begin{equation*}
 Res^I_N\phi \cong  m_{\phi}\sigma.     
\end{equation*}
 \end{proof}

In this case, $\phi$ can be termed as the \textbf{Clifford correspondent} of $Ind^G_{I(\sigma)}\phi$ over $\sigma$.
\begin{remark}
If $I_G(\sigma)=G$, i.e., $^{g}\sigma \sim 
\sigma$ for all $g \in G$, then for any $\theta\in \widehat{G}(\sigma)$
\begin{equation*}
    \ell = dimHom_N(\sigma, Res^G_N\theta)=1 \implies Res^G_N \theta = \sigma.
\end{equation*}
\end{remark}

\noindent Let $\widehat{G/N}$ denotes the set of all equivalence classes of irreducible modular representations of $G/N$.

\begin{defn}
The \textbf{inflation} $\bar{\psi} \in \widehat{G}$ of $\psi \in \widehat{G/N}$ is defined by setting 
\begin{center}
    $\bar{\psi}(g)=\psi(gN)$
\end{center}
for all $g \in G$.
\end{defn}
\begin{rmk}
    If  $\psi \in \widehat{G/N}$  is irreducible, then $\bar{\psi} \in \widehat{G}$ is irreducible.
\end{rmk}    
\begin{rmk}    
    For $\psi, \psi_1, \psi_2 \in \widehat{G/N}$, such that $\psi=\psi_1 \oplus \psi_2$, we have $\bar{\psi} =\bar{\psi_1}\oplus\bar{\psi_2}$.
\end{rmk}

\begin{rmk}
    If $p \nmid [G:N]$, then the irreducible modular representations of $G/N$ are completely reducible. We will be interested in the case when $p | [G:N]$, as this affects the semisimplicity of the representation space of $G/N$.
\end{rmk}

\begin{theorem}
\label{thm:inf}
Let $N \unlhd G$ be such that $G/N$ is a $p$-group. Let $\theta \in \widehat{G}$  and $W_{\theta}$ be the representation space of $\theta$ over $\mathbb{F}$. Suppose $\sigma \in \widehat{N}$ is such that $V_\sigma := Res^G_N W_{\theta}$. If $\psi \in \widehat{G/N}$ is a modular representation of ${G/N}$, then
\begin{equation}
    Ind^G_N\sigma=[G:N](\theta \otimes \bar{\psi}).
\end{equation}
Moreover, if $Res^G_N W_\tau=V_\sigma $ for $\tau \in \widehat{G}$, then $\tau \sim \theta \otimes \bar{\psi}$.
\end{theorem}

\begin{proof}
Let $\iota_N \in \hat{N}$ be trivial and $\lambda$ be the regular representation of $G/N$. Then we have
\begin{equation*}
    Ind^G_N\sigma=Ind^G_N(\sigma\otimes\iota_N)=Ind^G_N(Res^G_N\theta\otimes\iota_N)=\theta \otimes \bar{\lambda}
\end{equation*}
where $\bar{\lambda}$ is the inflation of $\lambda$.

Now by Prop 2.2.7 in \cite{Sch}, $G/N$ has a unique trivial $p$-modular representation, $\psi \in \widehat{G/N}$, upto isomorphism, i.e., $dim\psi=1$. Hence, if $m_{\psi}$ is the multiplicity of $\psi$, we have 
\begin{equation*}
    \lambda=\bigoplus_{\psi\in \widehat{G/N}}m_{\psi} \psi \implies   \lambda=[G:N]\psi \implies \bar{\lambda}=[G:N]\bar{\psi}.
\end{equation*}
Therefore,
\begin{equation*}
    Ind^G_N \sigma=\theta\otimes([G:N]\bar{\psi})=[G:N](\theta\otimes\bar{\psi}).
\end{equation*}

Now, for any $\tau \in \widehat{G}$ such that $Res^G_N W_\tau= V\sigma$, lemma \ref{Lem:Nak} implies that 
\begin{equation*}
    W_\tau \subseteq Ind^G_N V_\sigma \implies \tau \leq [G:N](\theta \otimes \bar{\psi}) \implies \tau \sim \theta \otimes \bar{\psi}.
\end{equation*}

\end{proof}

\begin{rmk}
The above theorem in terms of Brauer characters, is known as Green's theorem as stated in theorem 8.11 of \cite{navcharacters}.
\end{rmk}

Furthermore, we expect the following two cases to be true:
\begin{enumerate}[label = \roman*]

\item Let $N \unlhd G$ be such that $G/N$ is cyclic. Suppose $\sigma \in \widehat{N}$ is such that $I_G(\sigma) = G$.  Then there exists $\theta \in \widehat{G}$  such that $Res^G_N W_\theta =V_\sigma $, where  $W_{\theta}$ is the representation space of $\theta$ over $\mathbb{F}$. 

\item 
Let $N \unlhd G$ be such that $(|N|, |G:N|) =1$. Suppose $\sigma \in \widehat{N}$ is such that $I_G(\sigma) = G$. Then there exists $\mu \in \widehat{G}$  such that $Res^G_N W_\mu =V_\sigma $, where  $W_{\mu}$ is the representation space of $\mu$ over $\mathbb{F}$. 

\end{enumerate}
\begin{remark}
In the case of Brauer characters, the result ii. is due to E. C. Dade.
\end{remark}

\section{Clifford Correspondence between $SL_2(\mathbb{F}_p)$ and $GL_2(\mathbb{F}_p)$}

In this section, we first look at character-theoretic Cliiford theory for the modular representations of $SL_2(\mathbb{F}_3)$ and $GL_2(\mathbb{F}_3)$ as an example. We then proceed to develop some results for the modular representations of $SL_2(\mathbb{F}_p)$ and $GL_2(\mathbb{F}_p)$.

\subsection{Example: Modular Characters of $SL_2(\mathbb{F}_3)$ and $GL_2(\mathbb{F}_3)$}

In \cite{BN}, the modular characters are only defined for \textit{$p$-regular elements} of a group $G$  (i.e. those elements whose order is prime to $p$). If the group $G=G(\mathbb{F}_q)$ is an affine algebraic group over the finite field $\mathbb{F}_q$ and $q=p^r$, then $p$ is the \emph{defining} characteristic and the most relevant prime for our case. Therefore, for $SL_2(\mathbb{F}_3)$, we will consider $p=3$.

Moreover, by Theorem 2 of \cite{BN}, the number of irreducible Brauer characters of $SL_2(\mathbb{F}_3)$ over the field with characteristic $3$ equals the number of conjugacy classes of $3$-regular elements. These conjugacy classes of $3$-regular elements are represented by the identity matrix $I_2$, as well as $-I_2$, and the matrix $c_4(z) = \begin{pmatrix}
                  x & \Delta y \\
                  y & x \\
                  \end{pmatrix}$,
where $z=x+\sqrt{\Delta} y \in Ker(\mathbb{F}_9 \rightarrow \mathbb{F}_3)$, $z\neq \pm 1$, and $\Delta \in \mathbb{F}^*_3 - \mathbb{F}^*_9$.

The $3$-modular representations are summarized in Table 1.

\begin{table}[h!]
\caption{Modular Character Table of $SL_2(\mathbb{F}_3)$}\label{tab:title}
\begin{center}
    \begin{tabular}{|c|c|c|c|}  \hline   
$g$ & $I_2$ & $-I_2$ & $c_4(z)$\\ \hline
$|Cl_{SL_2(\mathbb{F}_3)}(g)|$ & 1 & 1 & 6 \\ \hline
$o(g)$ & 1 & 2  & 4 \\ \hline
$\sigma_0$ & $1$ & $1$ & $1$ \\ \hline
$\sigma_1$ & 2 & -2 & 0 \\ \hline
$\sigma_2$ & 3 & 3 & -1 \\ \hline
\end{tabular}
\end{center}
\end{table}


We will invoke the Clifford correspondence after inspection of the modular representations of $GL_2(\mathbb{F}_3)$. The conjugacy classes of the $3$-regular elements are represented by the identity matrix $I_2$, $-I_2$, $c_3(1, -1)=  \begin{pmatrix}
                 1  & \\
                     & -1 \\
                  \end{pmatrix}$,
                   $c_4(z)$, $c_4(-z)=-c_4(z)$ and the matrix $c_4(z^2)$. The modular representations corresponding to the six 3-regular conjugacy classes of $GL_2(\mathbb{F}_3)$ are summarized in Table 2.

\begin{table}[h!]
\caption{Modular Character Table of $GL_2(\mathbb{F}_3)$}\label{tab:title}
\begin{center}
    \begin{tabular}{|c|c|c|c|c|c|c|}  \hline   
$g$ & $I_2$ & $-I_2$ & $c_3(1,-1)$ & $c_4(z)$ & $c_4(-z)$ & $c_4(z^2)$\\ \hline
$|Cl_{SL_2(\mathbb{F}_3)}(g)|$ & 1 & 1 & 12 & 6 & 6 & 6 \\ \hline
$\theta_{0,0}$ & $1$ & $1$ & $1$ & $1$ & $1$ & $1$ \\ \hline
$\theta_{0,1}$ & 1 & 1 & -1 & -1 & -1 & 1 \\ \hline   
$\theta_{1,0}$ & 2 & -2 & 0 & $\sqrt{2}i$ &  $-\sqrt{2}i$ & 0 \\ \hline
$\theta_{1,1}$ & 2 & -2 & 0 & $-\sqrt{2}i$ & $\sqrt{2}i$ & 0 \\ \hline
$\theta_{2,0}$ & 3 & 3 & 1 & -1 & -1 & -1 \\ \hline
$\theta_{2,1}$ & 3 & 3 & -1 & 1 & 1 & -1 \\ \hline

\end{tabular}
\end{center}
\end{table}


Then we have for all $k$, the inertia group, $I_G(\sigma_k)=G$. Therefore, we find that $\hat{G}(\sigma_0)=\{\theta_{0,0}, \theta_{0,1}\}$, $\hat{G}(\sigma_1)=\{\theta_{1,0}, \theta_{1,1}\}$, and $\hat{G}(\sigma_2)=\{\theta_{2,0}, \theta_{2,1}\}$. 

Moreover, by theorem \ref{thm: 23}, we see
\begin{eqnarray*}
Ind^G_N\sigma_0 &= & \theta_{0,0} \oplus \theta_{0,1},\\
Ind^G_N\sigma_1 &= & \theta_{1,0}\oplus \theta_{1,1},\\
Ind^G_N\sigma_2 &= & \theta_{2,0} \oplus \theta_{2,1}. \\
\end{eqnarray*}

\subsection{Modular Representations of $SL_2(\mathbb{F}_p)$ and $GL_2(\mathbb{F}_p)$}
Let  $\mathbb{F}$ be the algebraic closure of the finite field $\mathbb{F}_p$ with characteristic $p$. 

This section aims to determine the relation between the irreducible modular representations of $SL_2(\mathbb{F}_p)$ over  $\mathbb{F}$ and that of $GL_2(\mathbb{F}_p)$ using the modular Clifford theory. Some of the main sources of notations, definitions and results for this section are \cite{Bon}, \cite{BN}, \cite{DG}, and \cite{SR}.


Let $G$ denote the group $GL_2(\mathbb{F}_p)$, $N$ its normal subgroup $SL_2(\mathbb{F}_p)$, $B$ the subgroup of upper triangular matrices \big($\begin{smallmatrix}
                a & b\\
                0 & d
\end{smallmatrix}$\big) in $N$, and $U$ the subgroup of matrices of the form \big($\begin{smallmatrix}
                1 & a\\
                0 & 1
\end{smallmatrix}$\big). Note that, $U$ is a $p$-Sylow subgroup of $N$. Setting $ w= \big(\begin{smallmatrix}
                0 & -1\\
                1 & 0
\end{smallmatrix}\big)$, the \textit{Bruhat decomposition} of $N$ is the  partition of $N$ into double cosets for the action of $B$:
\begin{equation*}
    N=B \dot{\cup} UwB.
\end{equation*}

For $k \geq 0$, let $Pol_k$ denote the space of homogeneous polynomials of degree $k$ in two variables $x$ and $y$. The dimension of $Pol_k$ is  $dim\langle\{x^iy^{k-i}|0 \leq i \leq k\}\rangle = k+1$.  For all $g=\big(\begin{smallmatrix}
                a & b\\
                c & d
\end{smallmatrix}\big) \in N$ and $P(x,y) \in Pol_k$, let the representation $\rho$ be defined as
\begin{equation}
  \rho_{Pol_k}(g)\cdot P(x,y)=P(ax+cy, bx+dy). 
\end{equation}

Then we have the following result.

\begin{theorem}
There are $p$ distinct irreducible representations of $SL_2(\mathbb{F}_p)$, upto isomorphism.
\end{theorem}

\begin{proof}
By Theorem 2 of \cite{BN}.
\end{proof}

\begin{theorem} \label{thm:22}
For $0 \leq k \leq p-1$, the representations $Pol_k$ are all irreducible.
\end{theorem}
\begin{proof}
We choose $u=\big(\begin{smallmatrix}
                1 & 1\\
                0 & 1
\end{smallmatrix}\big) \in U$. On the basis $\{x^iy^{k-i}|0 \leq i \leq k\}$, $u$ is represented by the matrix $R_{k+1}$, whose $(i,j)^{th}$ entry is the binomial coefficient \big($\begin{smallmatrix}
                j-1\\
                i-1
\end{smallmatrix}$\big) mod p, i.e.
\begin{center}
    $R_{k+1}=\begin{pmatrix}
    1 & 1 & 1 &  &  & \\
      & 1 & 2 &  & \ast  &  \\
      &   &  1 & 3 & &\\
      & & & 1 & \ddots & \\
      & & & & \ddots & k \\
      & & & & & 1
    \end{pmatrix}$
\end{center}
Note that the matrix $R_{k+1}-I_{k+1}$ has rank $k$. Then, by the rank - nullity theorem, the fixed points of $R_{k+1}$ form a subspace of dimension 1. Notice that all irreducible representations of $U$ are $1$-dimensional, since $U$ is abelian, and $U$ acts trivially on them. Therefore, $Pol_k|_U$, as a representation of $U$, has a unique minimal subrepresentation and so does any quotient space of $Pol_k|_U$. Hence, we have the unique uniserial composition series
\begin{equation*}
    \langle x^k\rangle < \langle x^k, x^{k-1}y \rangle < \cdots < \langle x^k, x^{k-1}y, \cdots, y^k \rangle = Pol_k|_U
\end{equation*}

Moreover, every non-zero subrepresentation of $Pol_k|_U$ contains $x^k$, but no proper subrepresentation contains $y^k$. Further,  we see, $x^k \cdot w \mapsto y^k$. Hence, no subrepresentation of $Pol_k$ contains $x^k$ without containing $y^k$, i.e, $Pol_k$ is irreducible.    
\end{proof}


Let $Pol_k(r)$, $r \in \mathbb{Z}$, be the same space of polynomial such that $g= \big(\begin{smallmatrix}
                a & b\\
                c & d
\end{smallmatrix}\big) \in G$ acts on a polynomial $P(x,y)$ as:
\begin{equation*}
    \rho_{Pol_k(r)} (g) \cdot P(x,y)= P(ax+cy, bx+dy)\otimes det(g)^r.
\end{equation*}
The following statement is proved in \cite{SR} for  $GL_2(\mathbb{F}_q)$, but we will restate it for $Gl_2(\mathbb{F}_p)$.

\begin{theorem} \label{thm: 23}
The irreducible representations of $GL_2(\mathbb{F}_p)$ are exactly the representations $Pol_k(r)$ with $0 \leq k \leq p-1$ and $0 \leq r \leq  p-2$.
\end{theorem}

By definition of the representation space $Pol_k$,  $k \in \{0,1, \dots p-1\}$, for any $P(x,y) \in Pol_k$, we have $\rho_{Pol_k} (g) \cdot P(x,y) \in Pol_k$, for all $g \in G$. Therefore, the inertia group of  the representation $Pol_k$ is 
\begin{equation*}
I_G(Pol_k)=G
\end{equation*}
and by definition \ref{defn: Ghat} we have 
\begin{equation} \label{eq: 2}
    \hat{G}(Pol_k)=\{Pol_k(0), \dots, Pol_k(p-2)\}
\end{equation}
Furthermore, we have $[G:N] =(p-1)$. Hence for $0 \leq k \leq p-1$, theorem \ref{thm:1}(1) implies that
\begin{equation} \label{eq:1}
Res^G_N (Ind^G_N Pol_k) =(p-1)Pol_k,
\end{equation}
and theorem \ref{thm:1}(3) as well as the Clifford correspondence implies that
\begin{equation} \label{eq:2}
Res^G_N Pol_k(r) = Pol_k.
\end{equation}

Finally, by lemma \ref{lem:1} we have the following decomposition,
\begin{equation} \label{eq:3}
Ind^G_N Pol_k =  \bigoplus_{r=0}^{p-2}Pol_k(r), \ 0 \leq k \leq p-1.
\end{equation}


\begin{small}
\noindent \textbf{Acknowledgements} The author is indebted to Kwangho Choiy for his suggestions and valuale comments.
\end{small}





\end{document}